\documentclass{article}
\usepackage[utf8]{inputenc}

\usepackage{amsfonts}
\usepackage{amsmath}
\usepackage{amssymb}
\usepackage{amsthm}
\usepackage{mathtools}
\usepackage{cleveref}

\newtheorem{theorem}{Theorem}
\newtheorem{lemma}[theorem]{Lemma}

\title{Solubility of Additive Forms of Twice Odd Degree over $\mathbb{Q}_2(\sqrt{5})$}
\author{Drew Duncan and David B. Leep}
\date{\today}

\begin{document}

\maketitle

\begin{abstract}
We prove that an additive form of degree $d=2m$, $m$ odd, $m\ge3$, over the unramified quadratic extension $\mathbb{Q}_2(\sqrt{5})$ has a nontrivial zero if the number of variables $s$ satisifies $s \ge 4d+1$.  If $3 \nmid d$, then there exists a nontrivial zero if $s \ge \frac{3}{2}d + 1$, this bound being optimal.  We give examples of forms in $3d$ variables without a nontrivial zero in case that $3 \mid d$.
\end{abstract}

\section{Introduction}
Consider the additive forms of degree $d$ over a p-adic field $K$ in $s$ variables:
\begin{equation}
\label{form}
a_1x_1^{d} + a_2x_2^{d} + \ldots + a_{s}x_{s}^{d}.
\end{equation}
Let $\Gamma^*(d, K)$ represent the minimum number such that any additive form of degree $d$ over $K$ in at least that many variables has a nontrivial zero.  A refinement of a famous conjecture by Artin holds that $\Gamma^*(d, K) \le d^2 + 1$ for any $p$-adic field $K$.  A number of papers have proven the $d^2 + 1$ holds when restricted to certain classes of degrees or $p$-adic fields, but so far none have succeeded in establishing it in general.  The seminal paper in this direction by Davenport and Lewis showed that the bound holds for all fields of $p$-adic numbers $\mathbb{Q}_p$, and that equality holds when $d = p - 1$.  In fact, a similar argument gives an additive form in $d^2$ variables with no nontrivial zero for any totally ramified $p$-adic field and $d = p - 1$.  It is not known if these are the only cases when $\Gamma^*(d, K) > d^2$.  This raises the interesting question of when the $d^2+1$ bound can be improved.

It was recently established in \cite{MGK} that $\Gamma^*(d, K) \le d^2 + 1$ for $K$ any of the seven quadratic extensions of $\mathbb{Q}_2$ and $d$ not a power of 2.  The present authors recently established the exact value for $\Gamma^*(4, K)$ where $K$ is one of the four ramified quadratic extensions $\mathbb{Q}_2(\sqrt{2})$, $\mathbb{Q}_2(\sqrt{10})$, $\mathbb{Q}_2(\sqrt{-2})$, or $\mathbb{Q}_2(\sqrt{-10})$  \cite{duncan2022quarticsolubility}, and a formula for the exact value of $\Gamma^*(2m, K)$ where $m$ is odd and $K$ is any of the six ramified quadratic extensions of $\mathbb{Q}_2$ \cite{2020arXiv201006833D}.  In this work, we investigate the analogous question for the unramified quadratic extension.

\begin{theorem}
Let $d = 2m$, with $m$ any odd number at least 3.  Then the following statements hold:
\begin{enumerate}
    \item $\Gamma^*(d, \mathbb{Q}_2(\sqrt{5})) = \frac{3}{2}d + 1$, if $3 \nmid d$,
    \item $3d+1 \le \Gamma^*(d, \mathbb{Q}_2(\sqrt{5})) \le 4d + 1$, if $3 \mid d$.
\end{enumerate}
\end{theorem}

Part one is proven in Theorems \ref{ndivupper} and \ref{ndivlower}, and part two is proven in Theorems \ref{divupper} and \ref{divlower}.  We note that this improves on the bound $\Gamma^*(6, \mathbb{Q}_2(\sqrt{5})) \le 29$ given in \cite{miranda_dissertation}.  We also conjecture that in fact $\Gamma^*(d, \mathbb{Q}_2(\sqrt{5})) = 3d+1$ when $3 \mid d$.  The lower bound in the case $d=6$ was established by an example credited to Knapp in \cite[page 69]{miranda_dissertation}.

% Let $\mathcal{F}(x, y, z) = x^6 + y^6 + z^6$.  Then,

% $$\mathcal{F}(x_1, x_2, x_3) + 2\alpha\mathcal{F}(x_4, x_5, x_6) + 2^2(1 + \alpha)\mathcal{F}(x_7, x_8, x_9) + $$
% $$2^3\mathcal{F}(x_{10}, x_{11}, x_{12}) + 2^4\alpha\mathcal{F}(x_{13}, x_{14}, x_{15}) + 2^5(1+\alpha)\mathcal{F}(x_{16}, x_{17}, x_{18}),$$
% where $\alpha = \frac{1 + \sqrt{5}}{2}$, is a form of degree $d=6$ in $3d$ variables with no nontrivial zero in $\mathbb{Q}_2(\sqrt{5})$.

% Indeed, let $\mathcal{F}(x, y, z) = x^d + y^d + z^d$.  If $d$ is divisible by 3, then
% \begin{align*}
% \sum_{i=0}^{\frac{d}{3}-1}
% \Big\lparen
% 2^{3i}\mathcal{F}(x_{9i}, x_{9i+1}, x_{9i+2}) + &
% 2^{3i+1}\alpha\mathcal{F}(x_{9i+3}, x_{9i+4}, x_{9i+5}) \\ + 
% & 2^{3i+2}(1 + \alpha)\mathcal{F}(x_{9i+6}, x_{9i+7}, x_{9i+8})
% \Big\rparen
% \end{align*}
% is a form of degree $d$ in $3d$ variables with no nontrivial zero in $\mathbb{Q}_2(\sqrt{5})$.  

% Furthermore, we will show in \Cref{notdivisible} that when $d$ is not divisible by 3, there is an even stronger result.

% \begin{theorem}
% Let $d = 2m$, with $m$ any odd number at least 3, $3 \nmid m$.  Then
% $$\Gamma^*(d, \mathbb{Q}_2(\sqrt{5})) = \frac{3}{2}d + 1.$$
% \end{theorem}

% Indeed,
% $$\sum_{i=0}^{\frac{d}{2}-1} 2^{2i}\mathcal{G}(x_{3i}, x_{3i+1}, x_{3i+2})$$
% with $\mathcal{G}(x,y,z)=x^d + y^d + \alpha z^d$
% is a form of degree $d$ in $\frac{3}{2}d$ variables with no nontrivial zero in $\mathbb{Q}_2(\sqrt{5})$.  That the above forms are anistropic is proved in the last section.

\section{Preliminaries}
The ring of integers $\mathcal{O}$ of $\mathbb{Q}_2(\sqrt{5})$ is $\mathbb{Z}_2[\alpha]$, where $\alpha = \frac{1 + \sqrt{5}}{2}$, 2 is the uniformizer of $\mathcal{O}$, and the residue field is the field of four elements with coset representatives $\{0, 1, \alpha, 1 + \alpha\}$.
Without loss of generality, we may assume that the coefficients of (\ref{form}) are in $\mathcal{O}$.

Any coefficient $a$ of a variable $x$ can be written as $a = 2^\ell(a_0 + 2a_1 + 4a_2 + \ldots)$, with $a_i \in \{0, 1, \alpha, 1 + \alpha\}$ and $a_0 \ne 0$.  
We will refer to $\ell$ as the \textit{level} of the variable, and we'll think of the variables as being distributed among these levels.
Within a level, we'll refer to the value $a_0$ as the variable's {0-class}, $a_1$ as its \textit{1-class}, the pair $a_0,a_1$ as its \textit{0,1-class}, etc.
For the sake of brevity, we'll often refer to a variable's 0-class as simply its \textit{class}.
Note that these classes can be viewed as elements of the residue field, and so it will make sense to talk about performing arithmetic with them.

We utilize a notation to indicate the configuration of variables among levels and classes:  The notation $(s_0, s_1, s_2, \ldots)$ indicates a form with at least $s_0$ variables in level 0, at least $s_1$ variables level 1, etc., and at least zero variables in all levels not indicated.  We will refer to the criteria indicated by this notation as a \textit{type}, and a form satisfying the criteria will be said to be of that type.  To make the type notation more specific, we will specify that there are at least a certain number of variables in each of the classes in a level by stacking the numbers vertically.  (Note that the order in which the classes appear in the type does not indicate the value of the classes.)  However, it will sometimes be convenient to label the classes so that they can be referenced.  We label them from top to bottom as: $a$, $b$, and $c$.

By the change of variables $\pi^r x^d = \pi^{r-id}(\pi^i x)^d = \pi^{r-id}y^d$ for $i \in \mathbb{Z}$, we will consider the level of a variable modulo $d$ for the remainder of the paper.  Multiplying a form by $\pi$ increases the level of each variable by one, and does not affect the existence of a nontrivial zero.  Considering the levels of variables modulo $d$, applying this transformation any number of times effects a cyclic permutation of the levels.  For this reason, any two types that differ by a cyclic permutation of their parameters will be considered as the same type.  This is also useful for arranging the variables in an order which is more convenient, a process to which we will refer as \textit{normalization}. See Lemma 3 of \cite{davenport1963homogeneous} for a proof of the following Lemma.

\begin{lemma}
Given an additive form of degree $d$ in an arbitrary local field $K$, let $s$ be the total number of variables, and $s_i$ be the number of variables in level $i \pmod{d}$. By a change of variables, the form may be transformed to one with:
\begin{align}
\begin{split}
s_0 \ge \frac{s}{d},  
\end{split}
\begin{split}
s_0 + s_1 \ge \frac{2s}{d},
\end{split}
\begin{split}
\ldots,
\end{split}
\begin{split}
s_0 + \ldots + s_{d-1} = s.
\end{split}
\end{align}
\end{lemma}

Consider a collection of terms $a_{i_1}x_{i_1}^d + a_{i_2}x_{i_2}^d + \ldots + a_{i_n}x_{i_n}^d$.  Some assignment $x_{i_j} = b_{i_j}$ of values to the variables yields some value $b$.  If instead we assign $x_{i_j} = b_{i_j}y$, then the terms are replaced with the new term $by^d$.  We call this replacement a \textit{contraction} of the original variables.
The level of the resulting variable is at least as high as that of any variable used in the contraction.
We proceed by showing that given a configuration of variables enough contractions can be performed so that a variable is produced which is at a sufficiently high level relative to the levels of the variables used to produce it.  A nontrivial zero then follows from the following version of Hensel's Lemma specialized to additive forms over p-adic fields.  (For a proof, see \cite{leep2018diagonal}.)

\begin{lemma}[Hensel's Lemma]
\label{quartic_hensels_lemma}
Let $x_i$ be a variable of (\ref{form}) at level $k$.  Suppose that $x_i$ can be used in a contraction of variables (or one in a series of contractions) which produces a new variable at level at least $k+3$.  Then (\ref{form}) has a nontrivial zero.
\end{lemma}

The sum of the classes of two variables in the same class is 0 in the residue field, and so we can contract such a pair at least one level up.
Similarly, three variables in the same level, one from each of the three nonzero classes, can be contracted at least one level up.
For the sake of brevity, the term \textit{pair} will henceforth refer to a pair in the same class, and the term \textit{triplet} will refer to three variables in a level distributed among the three nonzero classes, unless otherwise stated.
It will sometimes be helpful to contract two variables in the same level but in two different classes to a variable in the same level and in the third class.
(It is interesting to note that the contraction of a triplet can be thought of as two contractions of pairs in this way.)
If there are two pairs that contract to the same class one level up, then the resulting variables can be contracted at least one more level up.
We will refer to the four such original variables as a \textit{quadruplet}.

We will say that two variables are \textit{complementary} if contracting them produces a new variable at least two levels higher.
This occurs precisely when the two variables are in the same level, the same class $c_0$, and the sum of their 1-classes $c_1$ and $d_1$ is $c_0$, i.e., when $d_1 = c_1 + c_0$.
When it is known which class is being discussed, a pair of  1-classes contained in it will also be called complementary if a pair of variables with those 1-classes would be complementary.

Note that $5 = (2\alpha - 1)^2$, and that $5^2 \equiv 1 \mod 8$.  Thus $(2\alpha - 1)^d = 5^m = 5^{2k + 1} \equiv 5 \mod 8$.
Thus, 1 and 1+4 are d\textsuperscript{th} powers modulo 8, and are therefore d\textsuperscript{th} powers in $\mathcal{O}$ by Hensel's Lemma.
Letting $c_1$ and  $c_2$ be 0 or 1, corresponding to the choice of d\textsuperscript{th} power, we have $(a_0 + 2a_1 + 4a_2 )(1 + 4c_1) + (b_0 + 2b_1 + 4b_2)(1 + 4c_2) \equiv (a_0 + b_0) + 2(a_1 + b_1) + 4(a_2 + b_2 + c_1a_0 + c_2b_0) \pmod{8}$.  It follows that if we perform a contraction with variables in level $k$ and (possibly after a series of further contractions involving these variables) we have a variable in level $k+2$, then we may alter the class of this variable by adding to it the class of any of the originating variables from level $k$.
For example, if the resulting variable in level $k+2$ has class $a_0$, then by adding $a_0$ we can arrange instead for the contraction to produce a variable in level at least $k+3$.
We will refer to this ability to alter the class of the resulting variable by saying that it is free in some class (or multiple classes) at some level.

Because a zero follows from Hensel's Lemma if a contraction produces a variable at least three levels up, we will henceforth always assume that a contraction produces a new variable either one or two levels up.

\begin{lemma}
\label{triplet}
Suppose a triplet in level $k$ can be contracted, or used in a series of contractions, to produce a variable in level $k+2$.  Then (\ref{form}) has a nontrivial zero.
\end{lemma}
\begin{proof}
The resulting variable is free at level $k+2$ in all classes and so we can arrange the contraction so that the new variable is in level at least $k+3$.  A zero follows from Hensel's Lemma.
\end{proof}

Because of this, we will always assume that a triplet contracts just one level up.

\begin{lemma}
\label{twopairs1}
Suppose two variables can be produced in level $k+2$ by performing contractions such that for each of the variables at least one of its originating variables comes from level $k$.  Then (\ref{form}) has a nontrivial zero.
\end{lemma}
\begin{proof}
If the two resulting variables are in the same class, then another contraction can be performed, and a zero follows.  Thus, assume they are in different classes.  Because they are free in some class in this level, the contractions can be arranged so that each new variable appears in two possible classes.  Therefore, either both variables can be made to go to the same class, forming a pair which can be contracted at least one level higher, or at least one of the variables goes to class 0, i.e., can be made to go at least one level higher.  In either case, a zero follows from Hensel's Lemma.
\end{proof}

\begin{lemma}
\label{twopairs2}
Suppose two pairs from different classes in level $k$ contract to the same class in level $k+1$, or contract to different classes and can be used to form a triplet.  Then (\ref{form}) has a nontrivial zero.
\end{lemma}
\begin{proof}
Assume the resulting pair or triplet in level $k+1$ contracts to level $k+2$.  The new variable is free in at least two different classes, and thus also in the sum of these two classes.  It is therefore free in all three classes and can be made to go up to at least level $k+3$.  A zero follows from Hensel's Lemma.
\end{proof}

\begin{lemma}
\label{contractions}
Here we collect a number of useful consequences.

\begin{enumerate}
    \item If there are at least three variables in a level, then either a pair or a triplet can be contracted up at least one level.
    \item If there are at least four variables in a level, then a pair can be contracted up at least one level.
    \item If there are at least five variables in a level, then either a pair can be contracted up exactly one level or the form has a nontrivial zero.
    \item If there are at least five variables in a class, then either a pair or a quadruplet can be contracted up at least two levels.
    \item If the variables in level $k+1$ occupy at least two classes, and a triplet can be formed in level $k$, then the form has a nontrivial zero.
    \item If the variables in level $k+2$ occupy at least two classes, and a pair in level $k$ can be contracted to level $k+2$, then the form has a nontrivial zero.
    \item If there are four variables in a level in the same 0,1-class and their 2-classes sum to 0 or to their 0-class, then the form has a nontrivial zero.
\end{enumerate}
\end{lemma}

\begin{proof}
\begin{enumerate}
    \item Suppose no class contains two variables.  Then, each class contains at least one variable.  This triplet can be contracted up at least one level.
    \item By the pigeonhole principle, there is a class containing at least two variables.
    \item If any class contains at least three variables, then either two variables lie in the same 1-class, or in a pair of noncomplementary 1-classes, and so the class must contain a noncomplementary pair which contracts up exactly one level.  Thus assume at least two classes contain two variables.  If both pairs contract up two levels, then a zero follows from Lemma \ref{twopairs1}.
    \item Assume the variables are contained in two noncomplementary 1-classes.  Either one of the 1-classes contains at least four variables, or both contain at least two.  These form two pairs which contract to the same class one level higher, and thence at least one more level.
    
    \item By Lemma \ref{triplet}, assume the triplet contracts to level $k+1$.
    If the resulting variable is formed in an occupied class, it is contained in a pair.
    Otherwise, it is contained in a triplet.  In either case, the resulting variable can be used in a contraction resulting in a variable at least one more level up, and a zero follows from Lemma \ref{triplet}.
    
    \item If the resulting variables goes to an occupied class, then a pair can be formed and a zero follows.  Otherwise, a triplet can be formed and a zero follows.
    
    \item The sum of their coefficients is $4a_0 + 8a_1 + 4b \equiv 4(a_0 + b) \pmod{8}$, where $b$ is 0 or $a_0$ by hypothesis.
    Because the resulting variable is free in $a_0$ at level $k+2$, the quadruplet can be contracted to level at least $k+3$.
\end{enumerate}
\end{proof}

\begin{lemma}
\label{sameclass}
A pair contracts to the same class one level up if and only if the variables are in the same 0,1-class.
\end{lemma}
\begin{proof}
Let the two coefficients be $a \equiv a_0 + 2a_1 + 4a_2$ and $b \equiv b_0 + 2b_1 + 4b_2$ modulo $8$.  Assume $a_0 = b_0$.  If $a_1 = b_1$, then $a + b \equiv 2(a_0 + 2(a_1 + a_2 + b_2)) \mod{8}$.  If $a_1 \ne b_1$, then $a_1 + b_1 \ne 0$ and so $a + b \equiv 2((a_0 + a_1 + b_1) + 2(a_2 + b_2)) \mod{8}$ which is in class $a_0 + a_1 + b_1 \ne a_0$. 
\end{proof}

\begin{lemma}
\label{amongthree}
Among three variables in the same class, a pair can be formed which contracts exactly one level up.  If a pair which contracts to the same class one level up cannot be formed, then pairs can be formed which contract to each other class one level up, as well as a pair which contracts at least two levels up.
\end{lemma}
\begin{proof}
If any two of the variables are in the same 1-class, they contract to a new variable one level up in the same class by Lemma \ref{sameclass}.
Otherwise, all three 1-classes are different, and there are pairs which can be contracted to any class one level up except the original class (including class 0, i.e., the pair contracts two levels up).
\end{proof}

\begin{lemma}
\label{fivein1class}
Suppose there are five variables in the same 0,1-class.  Then a zero follows.
\end{lemma}
\begin{proof}
By the pigeonhole principle, at least two of the variables are in the same 2-class, i.e., they are congruent modulo 8, and so they contract to the same 0,1-class one level up.  If there is another pair in the same 2-class, then both pairs contract one level up to the same 0,1-class.  From there they contract up one more level, again to the same class.  Being free in that class, they can be made to contract at least one level further, and a zero follows from Hensel's Lemma.  Thus assume the remaining three variables are all in different 2-classes, and so there are three possibilities for the sums of their 2-classes, and thus three possibilities for the 1-class of the variable resulting from their contraction.  Choose a pair so that the resulting variable is complementary to the previous resulting variable, and so the four variables together contract at least three levels higher.
\end{proof}

\begin{lemma}
\label{fiveinclass}
Suppose there are five variables in a level in the same class.  Then either there are two pairs among them which contract to distinct classes one level up, or a zero follows.
\end{lemma}
\begin{proof}
Let the class of the variables be $a$.  By Lemma \ref{fivein1class}, assume the variables do not lie all in one 1-class.  By the pigeonhole principle, the 1-class $a_1$ which contains the most variables contains at least two.  If there is a variable in both the complementary 1-class $b_1$ and some noncomplementary 1-class $c_1$, then because the 1-classes $b_1$ and $c_1$ are themselves noncomplementary, the pair in $a_1$ and the pair formed from $b_1$ and $c_1$ go to different classes ($a$ and $a + b_1 + c_1$, respectively) one level up.  If there is a variable in both of the noncomplementary classes $c_1$ and $d_1$, then pairs formed from $a_1$ and $c_1$, and $a_1$ and $d_1$ go to different classes ($a + a_1 + c_1$ and $a + a_1 + d_1$, respectively) one level up.  Thus assume the variables are distributed among exactly two 1-classes.

Suppose there are exactly three variables in $a_1$.  If the other two are in a complementary 1-class, two complementary pairs can be contracted two levels up, and a zero follows from Lemma \ref{twopairs1}.  If the other two are in a noncomplementary 1-class, say $c_1$, then a noncomplementary pair from $a_1$ and a noncomplementary pair formed from $a_1$ and $c_1$ contract to different classes ($a$ and $a + a_1 + c_1$, respectively) one level up.

Thus, assume there are four variables in $a_1$.  If the remaining variable is in a noncomplementary class, say $c_1$, then a pair from $a_1$ and a pair formed from $a_1$ and $c_1$ contract to different classes ($a$ and $a + a_1 + c_1$, respectively) one level up.  Thus assume the remaining variable is in the complementary 1-class $b_1$.

By Lemma \ref{contractions} (7), assume that the 2-classes of the four variables in $a_1$ do not sum to 0 or $a$.  Thus, assume they do not all belong to the same 2-class, they are not distributed among all four of the 2-classes, and they are not distributed two each among two 2-classes.  If they belong to three different 2-classes, then consider the sum (modulo 8) of coefficients of a variable from $a_1$ and a variable from $b_1$, $$(a + 2a_1 + 4c) + (a + 2b_1 + 4b_2) \equiv 2(a + a_1 + b_1 + 2(c + b_2)) \equiv 4(a + c + b_2 + d),$$
where $c$ denotes the choice of 2-class and $a_1 + b_1 \equiv a + 2d \mod 4$.
We may therefore choose c to be either $b_2 + d$ or $a + b_2 + d$, and because the resulting variable is free in $a$, a zero follows.  Thus assume the four variables are distributed three in one 2-class, and one in another.  If the two 2-classes differ by $a$, then the sum of the 2-classes of the four variables is $a$, and a zero follows from Lemma \ref{contractions} (7). Thus a pair of complementary variables can be formed which contracts to two possible classes at least two levels up which don't differ by $a$.  Because $b$ and $c$ differ by $a$, one of the possible classes is either 0 or $a$, and because the resulting variable is free in $a$, a zero follows.
\end{proof}

In the next section we will show the intervening result that eight variables in a single level is sufficient to guarantee the existence of a nontrivial zero.

\section{One Level}

\begin{lemma}
\label{223}
If (\ref{form}) is of type $\left(\substack{2\\2\\3}\right)$, then it has a nontrivial zero.
\end{lemma}
\begin{proof}
A pair can be formed in each of the classes.
If any two of these pairs contract to the same class one level up, or all three contract to distinct classes one level up, a zero follows from Lemma \ref{twopairs2}.
Thus assume at least one of the possible pairs contracts up exactly two levels.  By Lemma \ref{amongthree}, there is a pair in class $c$ that goes to level 1, and so assume there is a pair of complementary variables in some other class (say, class $a$).

It follows that the sum of the 1-classes of the variables in class $a$ is $a$, and that by Lemma \ref{amongthree} there is a pair in $c$ that contracts to the same class one level up.  After such a contraction, a triplet can be formed with the remaining variables; suppose it contracts to class $d$, with $d \ne 0$ and $d \ne c$.  Then by changing the variable from $a$ which is used, another triplet can be formed which contracts to class $d + a$, and so we assume $d \ne a$. But then $d + a = c$, and so the triplet can be contracted to a class in level 1 containing a variable, and a zero follows from \Cref{triplet}.
\end{proof}

\begin{lemma}
If (\ref{form}) is of type $\left(\substack{1\\3\\3}\right)$, then it has a nontrivial zero.
\end{lemma}
\begin{proof}
By \Cref{amongthree}, there exists a pair in each of the classes $b$ and $c$ which can be contracted up one level to two distinct classes. The remaining variables form a triplet, and a zero follows from \Cref{contractions} (5).
\end{proof}

\begin{lemma}
If (\ref{form}) is of type $\left(\substack{1\\1\\5}\right)$, then it has a nontrivial zero.
\end{lemma}
\begin{proof}
By Lemma \ref{fiveinclass}, assume two pairs from the five variables in class $c$ can be contracted to two different classes one level up.  A triplet can be formed from the remaining variables, and a zero follows from Lemma \ref{contractions} (5).
\end{proof}

\begin{lemma}
If (\ref{form}) is of type $\left(\substack{0\\4\\4}\right)$, then it has a nontrivial zero.
\end{lemma}
\begin{proof}
By Lemma \ref{amongthree}, there is a pair in $b$ which contracts to some class in level 1, say $d$.  Similarly, there is a pair in $c$ which contracts to $e$ in level 1.  By Lemma \ref{twopairs2}, assume that there is at least one class in level 1 to which no pair contracts, say $f$, and that no pair from $b$ contracts to $e$ or from $c$ to $d$.  If both pairs from $b$ contract to $d$ and both from $c$ contract to $e$, then the two resulting pairs in level 1 can be contracted up at least one more level and a zero follows from Lemma \ref{twopairs1}.  Thus assume at least one pair from one of the classes contracts to level 2.  If the two pairs from the other class contract to the same class in level 1, then the resulting pair can be contracted up at least one more level and a zero again follows from Lemma \ref{twopairs1}.  Thus assume a pair from each of the classes contracts to level 2.  A zero follows again from Lemma \ref{twopairs1}.
\end{proof}

\begin{lemma}
If (\ref{form}) is of type $\left(\substack{0\\2\\5}\right)$, then it has a nontrivial zero.
\end{lemma}
\begin{proof}
By Lemma \ref{fiveinclass}, assume there are two pairs in $c$ which contract to two different classes in level 1.
If the pair in $b$ contract to level 1, then a zero follows from Lemma \ref{twopairs2}.
Thus assume the pair in $b$ contracts to level 2.
By Lemma \ref{contractions} (4), there is a pair or a quadruplet in $c$ which contracts to level 2.
A zero follows from Lemma \ref{twopairs1}.
\end{proof}

\begin{lemma}
\label{007}
If (\ref{form}) is of type $\left(\substack{0\\0\\7}\right)$, then it has a nontrivial zero.
\end{lemma}
\begin{proof}
If a pair in $c$ contracts to level 2, then by Lemma \ref{contractions} (4), another pair or quadruplet contracts to level 2, and a zero follows from Lemma \ref{twopairs1}.
Thus, assume the seven variables are distributed among two noncomplementary 1-classes, and so three pairs can be formed, each having both variables in the same 1-class.
Thus the three pairs can be contracted to three variables in level 1 in class $c$.
By Lemma \ref{amongthree}, either there is a pair among these three that contracts to a variable in class $c$ in level 2 which is free in class $c$ and a zero follows, or there is a pair which contracts at least two levels higher, and a zero follows from Hensel's Lemma.
\end{proof}

It is interesting to note that Lemmas \ref{223} through \ref{007} are independent of each other in the sense that no one relies on the truth of any of the others.

\begin{lemma}
\label{eight}
If any level in (\ref{form}) has at least eight variables, then it has a nontrivial zero.
\end{lemma}
\begin{proof}
Every distribution of eight variables among three classes is covered by at least one of the cases considered in Lemmas \ref{223} through \ref{007}.
\end{proof}

\begin{theorem}
$\Gamma^*(2m, \mathbb{Q}_2(\sqrt{5})) \le 7d + 1$.
\end{theorem}
\begin{proof}
By the pigeonhole principle, there is a level with at least eight variables.  The result follows from \Cref{eight}.
\end{proof}

Next, we consider types which constrain the number of variables in two consecutive levels.

\section{Two Levels}

\begin{lemma}
\label{0061}
If (\ref{form}) is of type $\left(\substack{0\\0\\6}, 1\right)$, then it has a nontrivial zero.
\end{lemma}
\begin{proof}
Let $d$ be the class of the variable in level 1.  If there is a pair in $c$ in level 0 that contracts to level 2, then by \Cref{twopairs1} assume the two remaining pairs contract to level 1, and so either a pair or a triplet contracts from there to level 2.  A zero follows from Lemma \ref{twopairs1}.  Thus, assume no pairs from $c$ contract to level 2.  It follows that the six variables are distributed among two noncomplementary 1-classes.  By Lemma \ref{fivein1class}, assume there are at most four variables in each 1-class.  If the variables are distributed four in one 1-class and two in the other, then there are three pairs that contract to $c$ in level 1, and by Lemma \ref{amongthree} either a pair contracts to a variable in $c$ which is free in $c$ or a pair contracts to at least level 3; in either case a zero follows.  Thus, assume they are distributed three in each of the two noncomplementary 1-classes, and so two pairs can be formed which contract to $c$ in level 1, and one pair which contracts to some other class, say $e$.  If $d = c$, then there are three variables in $c$, and a zero follows as before.  If $d=e$, then the pairs in each of $c$ and $d$ (in level 1) can be contracted to level 2, and a zero follows.  Thus assume $d \ne e$ and $d \ne c$.  We may thus contract a variable from each of $d$ and $e$ (in level 1) to a variable in $c$ (in level 1).  Then there are again three variables in $c$, and a zero follows as before.
\end{proof}

\begin{lemma}
\label{0241}
If (\ref{form}) is of type $\left(\substack{0\\2\\4}, 1\right)$, then it has a nontrivial zero.
\end{lemma}
\begin{proof}
Let $d$ be the class of the variable in level 1.  If any of the three pairs in level 0 contracts to level 2, then assume the other two pairs contract to level 1, and so a pair or triplet using one of the resulting variables can be contracted to level 2.  A zero follows from \Cref{twopairs1}.  Thus assume no pair in level 0 contracts to level 2.  If there are two pairs in $c$ that contract to different classes, then after contracting a pair from $b$, a zero follows from Lemma \ref{twopairs2}.  By Lemma \ref{amongthree}, there is a pair in $c$ that contracts to $c$ at level 1, and so assume all pairs in $c$ contract to $c$ at level 1.  If $d = c$, then there are three variables in $c$ at level 1, and by Lemma \ref{amongthree} either a pair contracts to a variable in $c$ at level 2 which is free in $c$ or a pair contracts to at least level 3; in either case a zero follows.  Thus assume $d \ne c$.  Let $e$ be the class to which the pair in $b$ contract to at level 1.  By Lemma \ref{twopairs2}, assume $e \ne c$.  If $e = d$, then contract the pairs in $d$ and $c$ in level 1 to level 2, and a zero follows from Lemma \ref{twopairs1}.  Thus assume $e \ne d$ and $e \ne c$.  A zero follows from Lemma \ref{twopairs2}.
\end{proof}

\begin{lemma}
\label{71}
If (\ref{form}) is of type $\left(7, 1\right)$, then it has a nontrivial zero.
\end{lemma}
\begin{proof}
By Lemmas \ref{223} through \ref{007}, we only need to examine types  $\left(\substack{0\\1\\6}, 1\right)$, $\left(\substack{0\\3\\4}, 1\right)$, and
$\left(\substack{1\\2\\4}, 1\right)$.  The conclusion follows from Lemmas \ref{0061} and \ref{0241}.
\end{proof}

\begin{lemma}
\label{0225}
If (\ref{form}) is of type $\left(\substack{0\\2\\2}, 5\right)$, then it has a nontrivial zero.
\end{lemma}
\begin{proof}
First, suppose we contract a pair from $b$ and $c$ to a variable in $a$ (in level 0).  Then a triplet can be formed, and so by Lemma \ref{contractions} (5), assume all of the variables in level 1 are in the same class, say $d$.  By Lemma \ref{fiveinclass}, assume two pairs from level 1 can be contracted to two different classes in level 2.  If a pair from level 0 contracts to level 2, a pair or triplet could be formed with the resulting variable, and a zero would follow.  Thus assume both pairs in level 0 contract to level 1.  By Lemma \ref{twopairs2}, assume they contract to different classes and that one of those classes is $d$.  After contracting two pairs from the preexisting variables in $d$ to two different classes in level 2, there is one preexisting variable and one new variable remaining in $d$.  Contract them to level 2, a pair or triplet can be formed with the resulting variable, and a zero follows.
\end{proof}

\begin{lemma}
\label{0045}
If (\ref{form}) is of type $\left(\substack{0\\0\\4}, 5\right)$, then it has a nontrivial zero.
\end{lemma}
\begin{proof}
Suppose one of the pairs in level 0 contracts to level 2.  Then by Lemma \ref{twopairs1} assume the other pair contracts to level 1, and if the variables in level 1 occupy at least two classes, a pair or triplet containing the new variable can be contracted to level 2, and a zero follows from Lemma \ref{twopairs1}.  So, if a pair contracts from level 0 to level 2, assume all of the variables in level 1 are in the same class, and by Lemma \ref{fiveinclass}, that there are two pairs that contract from level 1 to different classes in level 2, a pair or triplet can be formed with the variable originating from level 0, and a zero follows.  Thus assume both pairs in level 0 contract to level 1.

Suppose the preexisting variables in level 1 can be used to form two pairs which contract to distinct classes in level 2.  Then there would be one preexisting variable and two new variables in level 1 remaining, a pair or triplet containing a new variable could be contracted to level 2, and a zero would follow as before.  Thus by Lemma \ref{fiveinclass}, assume the five preexisting variables are not in the same class.

Suppose the preexisting variables in level 1 occupy all three classes.  If the pairs in level 0 contract to two distinct classes in level 1, then they can each be contracted with preexisting variables to level 2, and a zero follows from Lemma \ref{twopairs1}.  If the pairs in level 0 contract to the same class in level 1, then one can be contracted with a preexisting variable in the same class, and the other can be contracted with two preexisting variables in the other two classes, and a zero again follows from Lemma \ref{twopairs1}.  Thus assume the preexisting variables in level 1 occupy exactly two classes.  Call the one with larger number of variables $d$, the other $e$, and the remaining class $f$.

If a pair from level 0 contracts to $d$, then it can be used to form a pair which could be contracted to level 2.  At least one variable would remain in $d$, and so the remaining pair in level 0 could be contracted to level 1 and be used to form a pair or triplet which could be contracted to level 2, and a zero would follow from Lemma \ref{twopairs1}.  Thus assume all pairs from level 0 contract to $e$ or $f$ in level 1, and so the resulting variables can be used to form a pair which contracts to level 2.  The resulting variable is free in $c$ at level 2, and so assume it is contracted to either class $a$ or $b$.  Thus, if a pair from $d$ in level 1 contracts to $a$ or $b$ in level 2, a zero follows.  By Lemma \ref{amongthree}, at least one pair from $d$ in level 1 contracts to $d$ in level 2, and so assume $d=c$.  But by Lemma \ref{amongthree}, at least one pair from $c$ in level 0 contracts to $c$ in level 1.  A zero follows as before.
\end{proof}

\begin{lemma}
\label{55}
If (\ref{form}) is of type $\left(5, 5\right)$, then it has a nontrivial zero.
\end{lemma}
\begin{proof}
Suppose (\ref{form}) is of type $\left(\substack{1\\1\\3}, 5\right)$.  Contract a pair from $a$ and $c$ to $b$ (in the same level), or from $a$ and $b$ to $c$ (in the same level).
A zero follows from Lemma \ref{0225} or Lemma \ref{0045}, respectively.  For any other distribution of the variables in level 0, a zero follows from Lemma \ref{0225} or Lemma \ref{0045}.
\end{proof}

\begin{lemma}
\label{37}
If (\ref{form}) is of type $\left(3, 7\right)$, then it has a nontrivial zero.
\end{lemma}
\begin{proof}
By Lemma \ref{amongthree}, among the three variables in level 0, a pair can be contracted to level 1.  A zero follows from Lemma \ref{eight}.
\end{proof}

\begin{lemma}
\label{eleven}
If (\ref{form}) has eleven variables in two consecutive levels, a zero follows.
\end{lemma}
\begin{proof}
Every distribution of eleven variables among two consecutive levels is covered by one of Lemmas \ref{eight}, \ref{71}, \ref{55}, and \ref{37}.
\end{proof}

\begin{theorem}
$\Gamma^*(2m, \mathbb{Q}_2(\sqrt{5})) \le 5d + 1$.
\end{theorem}
\begin{proof}
By the pigeonhole principle (putting $10m+1$ variables in $m$ pairs of consecutive levels), there are two consecutive levels which contain eleven variables.  A zero follows from Lemma \ref{eleven}.
\end{proof}

\begin{lemma}
\label{541}
If (\ref{form}) is of type $\left(5, 4, 1\right)$, then it has a nontrivial zero.
\end{lemma}
\begin{proof}
Two contractions can be performed from level 0.
Either one of the resulting variables goes to level 2, or a contraction can be performed in level 1 using at least one of the resulting variables and at most one of the preexisting variables, producing a variable in level 2.
(Note that we may choose any of the preexisting variables to potentially be used in this contraction.)
By Hensel's Lemma, we assume the resulting variable goes to an unoccupied class.
The variable is free at level 2 in some class, say $a$.
Thus, assume that a variable can be produced in level 2 in either class $b$ or class $c$ and that the only occupied class in level 2 is $a$.

Further, assume that if a contraction using the variables in level 1 produces a variable in level 2, then that variable goes to class $a$.
Thus, if a triplet can be formed in level 1, then it either contracts directly to level 3, or to class $a$ in level 2 and thence to level 3, and a zero follows from \Cref{triplet}.  Similarly, suppose two pairs can be formed in level 1.
If at least one goes to level 3, the other can be contracted to level 3 and a zero follows from \Cref{twopairs1}.  If both pairs go to class $a$ in level 2, then a zero follows from \Cref{amongthree}.
Therefore, assume the variables in level 1 are distributed so that one class is empty, one class has exactly one variables, and the remaining class has exactly three.  By \Cref{amongthree}, assume the class containing three variables is $a$.

Now, by \Cref{contractions} (5), assume a triplet cannot be formed in level 0.
If two classes in level 0 contain at least two variables, then a pair of variables from each class can be contracted to the third class, and a triplet could be formed.
Thus, assume at least four variables are in class $a$, and so by \Cref{amongthree}, at least one pair contracts to class $a$ in level 1 or to level 2.
Another pair contracts from level 0 to level 2, or to level 1 and thence to level 2, and a zero follows from \Cref{twopairs1}.

%If all classes in level 1 are occupied or two of the classes contain at least two variables, then two variables can be produced in level 2 which originate from level 0, and a zero follows from Lemma \ref{twopairs1}.  Thus assume one of the classes in level 1 contains at least three variables.  Because a variable originating from level 0 can be produced in level 2, assume that every pair from this class contracts to the occupied class $a$ in level 2, and so by Lemma \ref{amongthree}, assume the class in level 1 which contains at least three variables is $a$.

\end{proof}

\begin{theorem}
\label{divupper}
$\Gamma^*(2m, \mathbb{Q}_2(\sqrt{5})) \le 4d + 1$.
\end{theorem}
\begin{proof}
By normalization, level 0 has at least five variables, and by Lemma \ref{eight}, at most seven.  By normalization, level 1 has at least one variable, and so by Lemma \ref{71} assume level 0 has five or six.  First suppose it has five.  Then by Lemma \ref{55} and normalization, level 1 has four.  By normalization, level 2 has at least one variable, and a zero follows from \Cref{541}.

%By Lemma \ref{contractions} (3), a pair from level 0 can be contracted to level 1, and so by Lemma \ref{55} and normalization, level 2 has four variables.  Continuing in this way, all levels $1, 2, \ldots, d-1$ contain exactly four variables.  Contracting a pair from level $d-2$ to level $d-1$ results in levels $d-1$ and $d \equiv 0 \mod d$ both containing at least five variables, and a zero follows from Lemma \ref{55}.

Thus suppose level 0 has six variables, and so by Lemma \ref{55} and normalization, level 1 has either three or four.  If it has four, then by normalization level 2 has at least one variables, and a zero follows from \Cref{541}.  Thus assume level 1 has three variables, and so by normalization that level 2 has at least four variables.
If each class in level 0 contains an even number of variables, then three pairs can be formed.  By \Cref{twopairs1} assume two of the pairs contract to level 1.  Otherwise, two classes in level 0 contain an odd number of variables.  If any class contains five variables, or two classes contain three variables, then by two applications of  \Cref{amongthree} two pairs can be contracted to level 1.  Otherwise, the number of variables among the classes are three, two, and one, and so from the class with three variables, a pair can be contracted to level 1 by \Cref{amongthree}, and a triplet can be formed from the remaining variables which can be contracted to level 1 by \Cref{triplet}.  In each case, this results in five variables in level 1.  If level 2 contains exactly four variables, then by normalization level 3 contains at least one, and a zero follows from \Cref{541},  and if level 2 contains at least five variables, a zero follows from \Cref{55}.

\end{proof}

It is interesting to note (though it is not proven here) that each of the types presented in the lemmas above is minimal in the sense that any type produced by reducing any of the numbers defining the type contains anisotropic forms.

\section{$3 \nmid d$}
\label{notdivisible}
We now turn our attention to the case when $d$ is not divisible by three.  Because the residue field is $\mathbb{F}_4$ and $3 \nmid d$, every element modulo $2$ is a $d$\textsuperscript{th} power.
Thus there is a change of variables which replaces any given variable with another in any nonzero class in the  same level.
In particular, it follows in effect that any two variables in a given level can be contracted up at least one level.

\begin{lemma}
\label{211}
If (\ref{form}) is of type $\left(2,1,1\right)$ and $3 \nmid d$, then it has a nontrivial zero.
\end{lemma}
\begin{proof}
Contract a pair from level 0.  Because levels 1 and 2 are occupied, subsequent contractions can be performed which result in a variable in level at least 3.
\end{proof}

\begin{lemma}
\label{31}
If (\ref{form}) is of type $\left(3,1\right)$ and $3 \nmid d$, then it has a nontrivial zero.
\end{lemma}
\begin{proof}
By a change of variables, the three variables in level 0 can be distributed among the three different classes, forming a triplet.  By \Cref{triplet}, assume the triplet contracts to level 1.  The resulting variable can be contracted one level higher, and zero follows from \Cref{triplet}.
\end{proof}

\begin{lemma}
\label{threecases}
If (\ref{form}) is of type
\begin{enumerate}
    \item $\left(5\right)$,
    \item $\left(4,0,1\right)$, or
    \item $\left(2,3\right)$,
\end{enumerate}
and $3 \nmid d$, then it has a nontrivial zero.
\end{lemma}
\begin{proof}
\begin{enumerate}
    \item By \Cref{contractions} (3), assume a pair contracts from level 0 to level 1.  A zero follows from \Cref{31}.
    \item If a pair from level 0 contracts to level 2, then it can be contracted with the preexisting variable to at least level 3, and a zero follows.  Thus assume a pair contracts from level 0 to level 1.  A zero follows from \Cref{211}.
    \item By a change of variables, we may assume that three variables in level 1 are in the same class.  By \Cref{amongthree}, assume that a pair contracts from level 1 to level 2.  A zero follows from \Cref{211}.
\end{enumerate}
\end{proof}

\begin{theorem}
\label{ndivupper}
$\Gamma^*(2m, \mathbb{Q}_2(\sqrt{5})) \le \frac{3}{2}d + 1$.
\end{theorem}
\begin{proof}
By normalization, level 0 contains at least two variables, levels 0 and 1 together contain at least four, and levels 0 through 2 contain at least five.  All possible distributions of these five variables among the three levels is covered by one of Lemmas \ref{211} through \ref{threecases}.
\end{proof}

\section{Lower Bounds}
\begin{theorem}
\label{ndivlower}
Let $d = 2m$, $m$ odd, $m \ge 3$. Then, $\Gamma^*(d, \mathbb{Q}_2(\sqrt{5})) \ge \frac{3}{2}d + 1$.  \end{theorem}
\begin{proof}
Let $\mathcal{G}(x, y, z) = x^d + y^d + \alpha z^d$.
We will show that 
$$\mathcal{H} = \sum_{i=0}^{\frac{d}{2}-1} 2^{2i}\mathcal{G}(x_{3i}, x_{3i+1}, x_{3i+2})$$
has no nontrivial zero in $\mathbb{Q}_2(\sqrt{5})$. Suppose that it does.  First, we note that the d\textsuperscript{th} powers modulo 4 are contained in 0, 1, $\alpha + 1$, and $\alpha + 2(\alpha + 1)$.  By direct calculation, one may see that $\mathcal{G}$ has no primitive zero modulo 4.  Therefore, each of $x_0$, $x_1$, and $x_2$ are divisible by 2.  After dividing through by 4, we again see that $x_3$, $x_4$, and $x_5$ are all divisible by 2.  Repeating this process, we conclude that all of the variables are divisible by 2, and so $\mathcal{H}$ has no primitive zero modulo $2^d$, and therefore no nontrivial zero in $\mathbb{Q}_2(\sqrt{5})$. \end{proof}

\begin{theorem}
\label{divlower}
Let $d = 2m$, $m$ odd, $m \ge 3$, $3 \mid m$. Then, $\Gamma^*(d, \mathbb{Q}_2(\sqrt{5})) \ge 3d + 1$.
\end{theorem}

\begin{proof}
Let $\mathcal{F}(x, y, z) = x^d + y^d + z^d$.
We will show that 
\begin{align*}
\mathcal{I} = \sum_{i=0}^{\frac{d}{3}-1}
\Big\lparen
2^{3i}\mathcal{F}(x_{9i}, x_{9i+1}, x_{9i+2}) + &
2^{3i+1}\alpha\mathcal{F}(x_{9i+3}, x_{9i+4}, x_{9i+5}) \\ + 
& 2^{3i+2}(1 + \alpha)\mathcal{F}(x_{9i+6}, x_{9i+7}, x_{9i+8})
\Big\rparen
\end{align*}
has no nontrivial zero in $\mathbb{Q}_2(\sqrt{5})$.  Suppose that it does.  First, we note that the d\textsuperscript{th} powers modulo 4 are 0 and 1.  Therefore, the values represented by $\mathcal{F}(x_0, x_1, x_2) + 2\alpha\mathcal{F}(x_3, x_4, x_5)$ modulo 4, are $\{0,1,2,3\} + 2\alpha\{0,1\}$.  In order to make this congruent to 0 modulo 4, we conclude that all of $x_0$, $x_1$, and $x_2$ are divisible by 2.  After dividing through by $2\alpha$, we again see that $x_3$, $x_4$, and $x_5$ are all divisible by 2.  Repeating this process, we conclude that all of the variables are divisible by 2, and so $\mathcal{I}$ has no primitive zero modulo $2^d$, and therefore no nontrivial zero in $\mathbb{Q}_2(\sqrt{5})$.  
\end{proof}

\bibliographystyle{plain}
\bibliography{biblio}

\end{document}